\newif\ifnummat
\newif\ifmoc      

\nummatfalse
\mocfalse

\moctrue

\ifnummat
    \RequirePackage{fix-cm}
    \documentclass{svjour3}
    \smartqed
    \DeclareMathAlphabet{\mathcal}{OMS}{cmsy}{m}{n}
\fi

\ifmoc
    \documentclass{amsart}
    \newtheorem{theorem}{Theorem}[section]
    \newtheorem{lemma}[theorem]{Lemma}

    \theoremstyle{definition}

    \newtheorem{algorithm}[theorem]{Algorithm}

    \theoremstyle{remark}
    \newtheorem{remark}[theorem]{Remark}

    \numberwithin{equation}{section}

\fi

\numberwithin{theorem}{section}
\numberwithin{equation}{section}
    \usepackage[all]{xy}
    \usepackage{amssymb,amsmath,stmaryrd}
    \usepackage{amscd}
    \usepackage{braket,amsfonts,amsopn}
    \usepackage{bm}
    \usepackage{makecell,multirow,diagbox}
    \usepackage{mathrsfs}
    \usepackage{graphicx,epstopdf}
    \usepackage{subfigure}
\ifmoc

\fi

\DeclareMathOperator{\divg}{div}
\DeclareMathOperator{\curl}{curl}
\DeclareMathOperator{\rot}{rot}
\DeclareMathOperator{\tr}{tr}
\DeclareMathOperator{\osc}{osc}

\newcommand{\Th}{\mathcal{T}_{h}}

\newcommand{\CT}{\mathcal{T}}
\newcommand{\CE}{\mathcal{E}}
\newcommand{\CR}{\mathcal{R}}

\newcommand{\lr}[1]{\llbracket#1\rrbracket}

\title{Quasi-optimal adaptive hybridized  mixed finite element methods for linear elasticity}

\def\myKeywords{linear elasticity, mixed finite element, hybridization, convergence, quasi-optimality}

\def\myAMS{65N12, 65N15, 65N30, 65N50}

\def\myAbstract{
For the planar Navier--Lam\'e equation in mixed form with symmetric stress tensors, we prove the uniform quasi-optimal convergence of an adaptive method based on the hybridized mixed finite element proposed in [Gong, Wu, and Xu: Numer.~Math.,
141 (2019), pp.~569--604]. The main ingredients in the analysis consist of a discrete a posteriori upper bound and a quasi-orthogonality result for the stress field under the mixed boundary condition. Compared with existing adaptive methods, the proposed adaptive algorithm could be directly applied to the traction boundary condition and be easily implemented.}

\begin{document}

\ifnummat
   \author{Yuwen Li}
  \institute{\yulShortAuthor : \yulAddress\, {Email:}\yulEmail}
  \date{Received: \today\  / Accepted: date}
  \maketitle
  \begin{abstract}\myAbstract\end{abstract}
  \begin{subclass}\myAMS \end{subclass}
  \markboth{
  \yulShortAuthor }{\shortTitle}
\fi


\ifmoc
    \bibliographystyle{amsplain}
    \author[Y.~Li]{Yuwen Li}
    \address{Department of Mathematics, The Pennsylvania State University, University Park, PA 16802}
    \email{yuwenli925@gmail.com}

    \subjclass[2010]{Primary \myAMS}
    \date{\today}
    \begin{abstract}\myAbstract\end{abstract}
    \maketitle
    \markboth{Y.~Li}{Quasi-optimal AMFEM for linear elasticity}
\fi



\section{Introduction}\label{sec1}
Adaptive finite element methods for numerical solutions of partial differential equations has been an active research area since 1980s. Using a sequence of self-adapted graded meshes, adaptive methods can achieve quasi-optimal convergence rate even for problems with singularity arising from, e.g., irregular data or domains with nonsmooth boundary. Convergence and optimality analysis of adaptive methods for symmetric and positive-definite elliptic problems has now reached maturity, see, e.g.,  \cite{Dorfler1996,MNS2000,BinevDahmenDeVore2004,Stevenson2007,CKNS2008,Verfurth2013,DieningKreuzerStevenson2016} and references therein. 

An important model problem in linear elasticity is the Navier--Lam\'e equation, which could be discretized by primal methods and mixed methods. Adaptive mesh refinement based on a posteriori error indicators is essential to deal with nonsmooth boundaries of elastic bodies in practice. For conforming elasticity elements, a robust error estimator  could be found in \cite{Carstensen2005}. In \cite{CR2012}, a quasi-optimal nonconforming adaptive Crouzeix--Raviart element method in primal form was developed under the pure displacement  boundary condition.

Compared with primal methods, mixed methods could easily handle the traction boundary condition and is more natural from a viewpoint of solid mechanics. Conservation of angular momentum is implied by the symmetry of stress tensors of elasticity equations in mixed form. However, mixed methods with strongly imposed symmetry usually leads to higher order polynomial shape functions \cite{AW2002,ArnoldAwanouWinther2008,HZ2014,Hu2015} and a priori error estimates relying on high solution regularity. In such situations, adaptivity is of great importance, see, e.g., \cite{CarstensenDolzmann1998,LonsingVerfurth2004,CGG2019,CHHM2018,LiZikatanov2020arXiv} for a posteriori error estimates of adaptive mixed finite element methods (AMFEMs) in linear elasticity. 

For second order elliptic equations in mixed form, theoretical analysis of AMFEMs is extensive, see, e.g., \cite{CarstensenHoppe2006,BeckerMao2008,CHX2009,HuangXu2012,FFP2014,HLMS2019,Li2019SINUM,Li2020MCOM}. 
The optimality result of  adaptive mixed methods for elasticity equations seems limited in the literature. One reason is that most finite elements for discretizing the symmetric stress tensors require $C^0$ vertex continuity. As a result, stress finite element spaces on nested meshes are not nested as spaces. Recently the work \cite{HuMa2021} develops a quasi-optimal AMFEM for mixed elasticity, based on a modified Hu--Zhang mixed element \cite{HZ2014} enriched carefully at each vertex on nested meshes. 
In the meantime, \cite{GWX2019} presents a hybridized mixed method for elasticity using complete piecewise polynomial stress space without any vertex degrees of freedom. In this paper, we shall adopt the hybridization strategy in \cite{GWX2019} and develop a quasi-optimal  adaptive mixed method for planar linear  elasticity, see \eqref{dismix} and Algorithm \ref{AMFEM}. Without specific treatment, our AMFEM could directly be applied to the (pure) traction boundary condition. Another advantage of this AMFEM is its easy implementation because no explicit continuous local basis is needed in hybridization, see Section \ref{secNE}.

The proposed AMFEM is designed to reduce the stress error $\|\sigma-\sigma_h\|_A$ with a convergence rate free of volumetric locking.  The framework of our analysis is similar to the convergence analysis of AMFEMs for Poisson's equation \cite{CHX2009,HuangXu2012}. However, the $C^1$ nodal space $W_h$ used in our analysis (see Lemma \ref{exact}) is much more complicated than $C^0$ nodal spaces and the well-known $C^1$ Argyris spaces. As a consequence, the regularized local interpolation onto $W_h$ is rather involved, especially when tailored to respect mixed boundary conditions. In addition, we present a detailed construction of the discrete a posteriori upper bound in Theorem \ref{disupper} for the stress error under general boundary conditions, which seems missing in the literature.

In the rest of this section, we introduce the continuous and discrete mixed formulations of the linear elasticity equation in $\mathbb{R}^2$. Let $\Omega$ be a simply connected polygonal domain. Let $\sigma$ and $u$ denote the stress and displacement fields produced by a body force acting on a linearly elastic body that occupies the region $\Omega\subset\mathbb{R}^2$. Then $u$ takes value in $\mathbb{R}^2$ and $\sigma$ takes value in $\mathbb{S}$, the space of symmetric $2\times2$ matrices. Given Lam\'e constants $\mu>0$, $\lambda>0$, define
\begin{align*}
    &\varepsilon (u)=\frac{1}{2}(\nabla u+(\nabla u)^\top),\\
    &\mathbb{C}\sigma=2\mu\sigma+\lambda\tr(\sigma)\delta,
\end{align*}
where $\tr$ denotes the trace of square matrices, and $\delta$ is the $2\times2$ identity matrix. The Navier--Lam\'e equation for planar elasticity reads
\begin{equation}\label{NL}
    \divg\big(\mathbb{C}\varepsilon (u)\big)=f,
\end{equation}
where $\divg$ is the divergence operator applied to each row of $\mathbb{C}\varepsilon(u)$. Given $\tau\in L^2(\Omega,\mathbb{S})$, the compliance tensor is defined as 
\begin{equation*}
    \mathbb{A}\tau=\mathbb{C}^{-1}\tau=\frac{1}{2\mu}\left(\tau-\frac{\lambda}{2\mu+2\lambda}(\tr\tau)\delta\right).
\end{equation*} Let $\overline{\partial\Omega}=\overline{\Gamma}_D\cup\overline{\Gamma}_N$ with relatively open subsets $\Gamma_D$, $\Gamma_N,$ and $\Gamma_D\cap\Gamma_N=\emptyset$. The part $\Gamma_N=\cup_{j=1}^J\Gamma_j$ is the disjoint union of several connected components $\{\Gamma_j\}_{j=1}^J$. Let $n$ be the outward unit normal to $\partial\Omega$.  We consider the mixed formulation of \eqref{NL} under the mixed boundary condition
\begin{equation}\label{ctsLE}
    \begin{aligned}
    &\mathbb{A}\sigma=\varepsilon (u),\\
    &\divg\sigma=f,\\
    &u=g_D\text{ on }\Gamma_D,\\
    &\sigma n=g_N\text{ on }\Gamma_N.
\end{aligned}
\end{equation}
Let $\mathcal{RM}$ be the space of rigid body motions
\begin{equation*}
    \mathcal{RM}=\{(c_1,c_2)^\top+c_3(-x_2,x_1)^\top: c_1,c_2,c_3\in\mathbb{R}\}.
\end{equation*} 
If $\Gamma_N=\emptyset,$  the load $f$ in \eqref{ctsLE} is required to satisfy the compatibility condition \begin{equation*}
    \int_\Omega f\cdot vdx=0,\quad\forall v\in\mathcal{RM}.
\end{equation*}

Given a vector space $\mathbb{V}$, let $L^2(\Omega,\mathbb{V})$ denote the space of $\mathbb{V}$-valued $L^2$-functions on $\Omega.$ Similarly $H^{s}(\Omega, \mathbb{V})$ is the $\mathbb{V}$-valued $H^s(\Omega)$ Sobolev space. Define the spaces
\begin{align*}
    &U:=L^2(\Omega,\mathbb{R}^2)\text{ if }\Gamma_N\neq\emptyset,\quad U:=L^2(\Omega,\mathbb{R}^2)/\mathcal{RM}\text{ if }\Gamma_N=\emptyset,\\
    &\Sigma(g):=\{\tau\in L^2(\Omega,\mathbb{S}): \divg\tau\in L^2(\Omega,\mathbb{R}^2), \tau n=g\text{ on }\Gamma_N\}.
\end{align*} 
Given a subdomain $\Omega_0\subseteq\Omega$, let $(\cdot,\cdot)_{\Omega_0}$ denote the $L^2(\Omega_0)$ inner product and $(\cdot,\cdot)=(\cdot,\cdot)_{\Omega}$. For a 1d submanifold $\Gamma_0\subseteq\overline{\Omega}$, by $\langle\cdot,\cdot\rangle_{\Gamma_0}$ we denote the $L^2(\Gamma_0)$ inner product. 
The variational formulation of \eqref{ctsLE} seeks $\sigma\in\Sigma(g_N)$ and $u\in U$ such that
\begin{subequations}\label{ctsmix}
    \begin{align}
    (\mathbb{A}\sigma,\tau)+(\divg\tau,u)&=\langle\tau n,g_D\rangle_{\Gamma_D},\quad\tau\in\Sigma:=\Sigma(0),\label{mix1}\\
    (\divg\sigma,v)&=(f,v),\quad v\in U.\label{mix2}
\end{align}
\end{subequations}

Let $\mathcal{T}_0$ be a conforming initial macro-triangulation of $\Omega$ and be aligned with $\Gamma_D$, $\Gamma_N$. Let $\mathbb{T}=\{\CT_h\}$ denote a forest of conforming refinement of $\mathcal{T}_0$ indexed by $h$. For $\Th, \CT_H\in\mathbb{T}$, we say $\CT_H\leq\Th$ provided $\CT_h$ is a refinement of $\CT_H.$ We assume $\mathbb{T}$ is shape regular, i.e., there exists a uniform constant $\gamma_0$ such that
\begin{align*}
    \max_{\Th\in\mathbb{T}}\max_{T\in\Th}r_T/\rho_T<\gamma_0<\infty,
\end{align*}
where $r_T$ and $\rho_T$ are radii of circumscribed and inscribed circles of $T$, respectively. 
Let $\mathcal{P}_r(T,\mathbb{V})$ denote the space of $\mathbb{V}$-valued polynomials of degree at most $r$ on $T$.  For an integer $r\geq0,$ the mixed finite element spaces are
\begin{align*}
    \Sigma_h(g)&:=\{\tau_h\in \Sigma(g): \tau_h|_T\in\mathcal{P}_{r+3}(T,\mathbb{S})~\forall T\in\mathcal{T}_h\},\\
    U_h&:=\{v_h\in U: v_h|_T\in\mathcal{P}_{r+2}(T,\mathbb{R}^2)~\forall T\in\mathcal{T}_h\}.
\end{align*}

In the sequel, 
we assume that $g_N$ is a piecewise polynomial on $\Gamma_N$ with $g_N|_e\in\mathcal{P}_{r+3}(e,\mathbb{R}^2)$ for each edge $e$ in $\mathcal{T}_0$ and $g_D$ is a piecewise polynomial on $\Gamma_D$ aligned with $\mathcal{T}_0$. 
The mixed method for \eqref{ctsmix} is to find $\sigma_h\in \Sigma_h(g_N), u_h\in U_h$ such that
\begin{subequations}\label{dismix}
    \begin{align}
    (\mathbb{A}\sigma_h,\tau_h)+(\divg\tau_h,u_h)&=\langle\tau_h n,g_D\rangle_{\Gamma_D},\quad\tau_h\in\Sigma_h:=\Sigma_h(0),\label{dismix1}\\
    (\divg\sigma_h,v)&=(f,v_h),\quad v_h\in U_h\label{dismix2}.
\end{align}
\end{subequations}
When $\CT_H\leq\Th$, it holds that $\Sigma_H\times U_H\subseteq\Sigma_h\times U_h$. Then using the nestedness and  \eqref{dismix}, we obtain the Galerkin orthogonality
\begin{subequations}
    \begin{align}
    (\mathbb{A}(\sigma_h-\sigma_H),\tau_H)+(\divg\tau_H,u_h-u_H)&=0,\quad\tau_H\in\Sigma_H,\label{error1}\\
    (\divg(\sigma_h-\sigma_H),v_H)&=0,\quad v_H\in U_H\label{error2}.
\end{align}
\end{subequations}
It has been shown in \cite{HZ2014,GWX2019} that $\Sigma_h\times U_h$ fulfills the inf-sup condition
\begin{equation}\label{infsup}
    \|v_h\|\leq C\sup_{0\neq\tau_h\in\Sigma_h}\frac{(\divg\tau_h,v_h)}{\|\tau_h\|+\|\divg\tau_h\|},\quad\forall v_h\in U_h,
\end{equation}
where $C$ dependes only on $r$ and $\gamma_0$. However, the construction of the local basis of $\Sigma_h$ is rather involved \cite{GWX2019}. To overcome this difficulty, \eqref{dismix} is implemented using hybridization technique and iterative solvers, see \cite{GWX2019} and Section \ref{secNE}.

Let $\mathcal{V}_h$ denote the set of grid vertices in $\mathcal{T}_h.$ For $r\geq0,$ the classic Arnold--Winther mixed elasticity element spaces are
\begin{align*}
    \Sigma^{\text{AW}}_h&:=\{\tau_h\in \Sigma: \tau_h|_T\in\mathcal{P}_{r+3}(T,\mathbb{S}),  \divg\tau\in\mathcal{P}_{r+1}(T,\mathbb{R}^2)~\forall T\in\Th,\\
    &\qquad\tau_h\text{ is continuous at each }x\in\mathcal{V}_h\},\\
    U^{\text{AW}}_h&:=\{v_h\in U: v_h|_T\in\mathcal{P}_{r+1}(T,\mathbb{R}^2)~\forall T\in\Th\}.
\end{align*}
The Hu--Zhang mixed elasticity element spaces are 
\begin{align*}
    &\Sigma^{\text{HZ}}_h:=\{\tau_h\in \Sigma: \tau_h|_T\in\mathcal{P}_{r+3}(T,\mathbb{S})~\forall T\in\Th,\\
    &\qquad\qquad\tau_h\text{ is continuous at each }x\in\mathcal{V}_h\},\\
    &U^{\text{HZ}}_h:=U_h.
\end{align*}
Due to the continuity constraint of $\Sigma_h^{\text{AW}}$ and $\Sigma_{h}^{\text{HZ}}$ at each vertex,  we note that $\Sigma_H^{\text{AW}}\not\subseteq\Sigma_h^{\text{AW}}$, $\Sigma_H^{\text{HZ}}\not\subseteq\Sigma_h^{\text{HZ}}$. This non-nestedness is the motivation of our analysis of adaptive hybridized MFEM and a major difficulty arising from the analysis of AMFEMs based on Arnold--Winther and Hu--Zhang elements.

The rest of this paper is organized as follows. In Section \ref{seccomplex}, we introduce preliminaries for deriving the discrete a posteriori upper error bound. In Section \ref{secsigma}, we derive the discrete reliability and quasi-orthogonality. Section \ref{secqo} is devoted to the convergence and optimality analysis of the proposed adaptive algorithm. In Section \ref{secdisapprox}, we give proofs of technical results used in our analysis. The numerical experiment is presented in Section \ref{secNE}.

\section{Preliminaries}\label{seccomplex}
For $T\in\Th,$ let $|T|$ denote the area of $T$ and $h_T=|T|^{\frac{1}{2}}$ the size of $T.$ On $\partial T$ let $t$ be the counterclockwise unit tangent and $n$ the outward unit normal to $\partial T$. On $\partial\Omega$ let $t$ be the counterclockwise unit tangent to $\partial\Omega$. In $\mathcal{T}_h$, let $\mathcal{E}_h$, $\mathcal{E}^o_h$, $\mathcal{E}^D_h$ denote the set of edges, interior edges, and edges in $\Gamma_D$, respectively. Let 
\begin{equation*}
\CE^o_h(T)=\{e\in\mathcal{E}^o_h: e\subset\partial T\},\quad\CE^D_h(T)=\{e\in\mathcal{E}^D_h: e\subset\partial T\}.     
\end{equation*}
We use $\|\cdot\|_{\Omega_0}$ to denote the $L^2(\Omega_0)$ norm  and $\|\cdot\|=\|\cdot\|_\Omega$. Each edge $e$ in $\Th$ is assigned with a unit tangent $t_e$ and a unit normal $n_e$. In addition, $t_e$ is counterclockwise oriented and $n_e$ is outward pointing provided $e\subset\partial\Omega$. If $e$ is an interior edge shared by two triangles $T_+$ and $T_-$, let $\lr{\phi}|_e=(\phi|_{T_+})|_e-(\phi|_{T_{-}})|_e$ denote the jump of $\phi$ over $e,$ where $n_e$ is pointing from $T_+$ to $T_-$. 

Given a scalar-valued function $w$ and  a $\mathbb{R}^2$-valued function $\phi=(\phi_1,\phi_2)$, let
\begin{align*}
&\curl w:=\left(-\frac{\partial w}{\partial x_2},\frac{\partial w}{\partial x_1}\right)^\top,\quad\rot\phi:=\frac{\partial\phi_2}{\partial x_1}-\frac{\partial \phi_1}{\partial x_2}.
\end{align*}
For $\mathbb{R}^2$-valued $v=(v_1,v_2)^\top$ and $\mathbb{R}^{2\times2}$-valued $\tau=(\tau_1,\tau_2)^\top,$ let
\begin{align*}
&\curl v:=(\curl v_1,\curl v_2)^\top,\quad\rot \tau:=(\rot\tau_1,\rot\tau_2)^\top.
\end{align*}
For a unit vector $d$, we use $\partial_d$ to denote the directional derivative along $d.$ 
The stress error will be estimated by  $\eta_h=\eta_h(\sigma_h)=\big(\sum_{T\in\Th}\eta_h^2(\sigma_h,T)\big)^\frac{1}{2}$ with the element-wise error indicator given as 
\begin{align*}
    &\eta_h(\sigma_h,T)=\big\{h_T^4\|\rot\rot \mathbb{A}\sigma_h\|^2_T+\sum_{e\in\CE^o_h(T)}\big(h_e\|t_e^\top\lr{\mathbb{A}\sigma_h}t_e\|^2_e\\
    &\quad+h_e^3\|n_e^\top\partial_{t_e}\lr{\mathbb{A}\sigma_h}t_e-\lr{\rot \mathbb{A}\sigma_h}\cdot t_e\|^2_e\big)+\sum_{e\in\CE^D_h(T)}\big(h_e\|t_e^\top((\mathbb{A}\sigma_h)t_e-\partial_{t_e}g_D)\|^2_e\\
    &\quad+h_e^3\|n_e^\top\partial_{t_e}(\mathbb{A}\sigma_h)t_e-(\rot \mathbb{A}\sigma_h)\cdot t_e-n_e\cdot\partial^2_{t_e}g_D\|^2_e\big)\big\}^\frac{1}{2},
\end{align*}
where $h_e$ is the diameter of $e.$
By $P_h$ we denote the $L^2$ projection onto $U_h$. The data oscillation is $\osc_h=\osc_h(f)=\big(\sum_{T\in\Th}\osc^2_h(f,T)\big)^\frac{1}{2},$ where
\begin{align*}
    \osc_h(f,T)&=h_T\|f-P_hf\|_T.
\end{align*}
The expression of $\eta_h$ is the same as existing a posteriori error estimators for the MFEM using the Arnold--Winther and Hu--Zhang elements, see, e.g., \cite{CGG2019,CHHM2018}.

An indispensable ingredient of optimality analysis of AFEMs is the discrete upper bound for the finite element error. 
To construct such a bound, we consider the $C^1$-conforming space
\begin{align*}
    W_h=&\{w_h\in C^1(\overline{\Omega}): w_h|_T\in \mathcal{P}_{r+5}(T)~\forall T\in\Th,\\
    &(\curl w_h)|_{\Gamma_j}\text{ is constant for }1\leq i\leq J\},
\end{align*}
which is a subspace of the  Morgan--Scott $C^1$ element space \cite{MS1975,GiraultScott2002}.
Let $J$ denote the Airy stress function:
\begin{equation*}
    J=\curl\curl=\begin{pmatrix}\frac{\partial^2}{\partial x_2^2}&-\frac{\partial^2}{\partial x_1\partial x_2}\\
-\frac{\partial^2}{\partial x_1\partial x_2}&\frac{\partial^2}{\partial x_1^2}\end{pmatrix}.
\end{equation*}
Due to $J(W_h)\subset\Sigma_h$ and $\divg\circ J=0$, we obtain a well-defined discrete sequence: 
\begin{align}\label{exactsequence}
\begin{CD}
W_h@>{J}>>\Sigma_h@>{\divg}>>U_h@>>>0.
\end{CD}
\end{align}
\begin{lemma}\label{exact}
The sequence \eqref{exactsequence} is exact, i.e., $\ker(\divg|_{\Sigma_h})=J(W_h)$.
\end{lemma}
\begin{proof}
Given $\tau_h\in {\Sigma_h}$ with $\divg\tau_h=0,$ there exists $\phi\in H^1(\Omega)$ such that $\curl\phi=\tau_h$.  Due to the symmetry of $\tau_h,$ it holds that $\divg\phi=0$ and thus $\phi=\curl w$ for some $w\in H^2(\Omega)$. Therefore we obtain $\tau_h=\curl\curl w$, $w|_T\in\mathcal{P}_{r+5}(T)$ for each $T\in\mathcal{T}_h.$ The boundary condition $\tau_h n|_{\Gamma_N}=0$ implies $\partial_t(\curl w)|_{\Gamma_N}=0$ and $\curl w$ is constant on each $\Gamma_j.$ The proof is complete.
\end{proof}

The next theorem is a direct consequence of Lemma \ref{exact}.
\begin{theorem}[discrete Helmholtz decomposition]\label{disdecomp}
\begin{align*}
    \Sigma_h=J(W_h)\oplus\varepsilon^h_{\mathbb{C}}(U_h),
\end{align*}
where $\varepsilon^h_{\mathbb{C}}: U_h\rightarrow\Sigma_h$ is the adjoint operator of $-\divg: \Sigma_h\rightarrow U_h$, i.e.,
\begin{align*}
    (\mathbb{A}\varepsilon_{{\mathbb{C}}}^h(v_h),\tau_h)=-(v_h,\divg\tau_h)\text{ for all }\tau_h\in\Sigma_h.
\end{align*}
\end{theorem}
\begin{proof}
Let $\ker(\divg|_{\Sigma_h})^\perp$ be the orthogonal complement of $\ker(\divg|_{\Sigma_h})$ in $\Sigma_h$ with respect to the weighted inner product $(\mathbb{A}\cdot,\cdot)$. Elementary linear algebra shows that
$$\ker(\divg|_{\Sigma_h})^\perp=\varepsilon_{\mathbb{C}}^h(U_h).$$
Combining it with the exactness $\ker(\divg|_{\Sigma_h})=J(W_h)$ in Lemma \ref{exact}, we obtain
\begin{align*}
    \Sigma_h=\ker(\divg|_{\Sigma_h})\oplus\ker(\divg|_{\Sigma_h})^\perp=J(W_h)\oplus\varepsilon^h_{\mathbb{C}}(U_h),
\end{align*}
which completes the proof.
\end{proof}
\begin{remark}
For the Arnold--Winther and Hu--Zhang elements under $\Gamma_N=\emptyset$, the correct discrete elasticity sequences are
\begin{align*}
\begin{CD}
\widehat{W}_h@>{J}>>\Sigma^{\emph{AW}}_h@>{\divg}>>U_h^{\emph{AW}}@>>>0,
\end{CD}
\end{align*}
and 
\begin{align*}
\begin{CD}
\widehat{W}_h@>{J}>>\Sigma^{\emph{HZ}}_h@>{\divg}>>U_h@>>>0,
\end{CD}
\end{align*}
respectively, where
\begin{align*}
    \widehat{W}_h&=\{w_h\in H^2(\Omega): W_h|_T\in\mathcal{P}_{r+5}(T)\text{ for each }T\in\Th,\\
    &\qquad \nabla^2w_h\text{ is continuous at each }x\in\mathcal{V}_h\}.
\end{align*}
When $r=0,$ $\widehat{W}_h$ is the well-known quintic Argyris finite element space. Due to the extra vertex continuity, it is relatively easy to construct a local basis of $\widehat{W}_h$ and interpolation onto $\widehat{W}_h$. 
\end{remark}

It is noted that $W_h$ is not a standard finite element space. The work \cite{MS1975} gives a set of unisolvent nodal variables and locally supported dual nodal basis of $W_h$, although those degrees of freedom are much more complicated than the Argyris-type $C^1$ space $\widehat{W}_h$. Based on a slightly modified (but complicated) nodal variables,  Girault and Scott \cite{GiraultScott2002} constructed a locally defined and $H^1$-bounded interpolation preserving the homogeneous boundary condition. To derive the discrete reliability of \eqref{dismix}, we present an interpolation $I_h: W_h\rightarrow W_H,$ which is a slight variation of the interpolation in \cite{GiraultScott2002}. Throughout the rest of this paper, we say $c_1\lesssim c_2$ provided $c_1\leq cc_2$ for some generic constant $c$ depending only on $\mu, \Omega,$ $\gamma_0$, $r$. 
\begin{lemma}\label{propIH}
For $\Th, \CT_H\in\mathbb{T}$ with $\CT_H\leq\Th$, let
${\mathcal{R}}_H:=\CT_H\backslash\CT_h$ be the set of refinement elements and 
$$\widetilde{\mathcal{R}}_H:=\{T\in\CT_H: T\cap T^\prime\neq\emptyset\text{ for some }T^\prime\in\CR_H\}$$
denote the enriched collection of refinement elements.
There exists an interpolation $I_H: W_h\rightarrow W_H$ such that for $w_h\in W_h,$
\begin{subequations}
    \begin{align}
w_h-I_Hw_h&=0\text{ at }x\in\mathcal{V}_H,\label{localV}\\
w_h-I_Hw_h&=0\text{ on }T\in\CT_H\backslash\widetilde{\CR}_H,\label{localT}\\
w_h-I_Hw_h&=0\text{ on }\Gamma_N,\label{GammaN}\\
\partial_n(w_h-I_Hw_h)&=0\text{ on }\Gamma_N.\label{dnGammaN}
\end{align}
\end{subequations}
In addition,
\begin{equation}\label{approxG}
    \begin{aligned}
    &\sum_{T\in\CT_H}h_T^{-4}\|w_h-I_Hw_h\|^2_T+h^{-2}_T|w_h-I_Hw_h|^2_{H^1(T)}\\
    &\quad+h_T^{-3}\|w_h-I_Hw_h\|^2_{\partial T}+h_T^{-1}\|\nabla(w_h-I_Hw_h)\|^2_{\partial T}\lesssim|w_h|^2_{H^2(\Omega)}.
\end{aligned}
\end{equation}
\end{lemma}
The proof of Proposition \ref{propIH} is postponed in Section \ref{secdisapprox}.

\section{Discrete reliability and quasi-orthogonality}\label{secsigma}
Let $\|\cdot\|_{\mathbb{A}}$ denote the norm corresponding to $(\mathbb{A}\cdot,\cdot)$. For $\mathcal{M}\subseteq\Th,$ let
\begin{align*}
    \eta_h(\sigma_h,\mathcal{M})&=\big(\sum_{T\in\mathcal{M}}\eta_h^2(\sigma_h,T)\big)^\frac{1}{2},\\
    \osc_h(f,\mathcal{M})&=\big(\sum_{T\in\mathcal{M}}\osc_h^2(f,T)\big)^\frac{1}{2}.
\end{align*}
We shall prove the discrete reliability of the estimator $\eta_h$ and quasi-orthogonality between $\sigma-\sigma_h$ and $\sigma_h-\sigma_H$. 
The analysis relies on the following discrete approximation result, whose proof is left in Section \ref{secdisapprox}.  
\begin{lemma}\label{disapprox}
Let $\Th, \CT_H\in\mathbb{T}$ with $\CT_H\leq\Th$ and $Q_H$ denote the $L^2$-projection onto the space of piecewise rigid body motions $$\mathcal{RM}_H=\{v\in L^2(\Omega,\mathbb{R}^2): v|_T\in\mathcal{RM}\text{ for all }T\in\CT_H\}.$$ It holds that
\begin{equation*} 
\big(\sum_{T\in\CT_H}h_T^{-2}\|v_h-Q_H v_h\|^2_{T}\big)^\frac{1}{2}\lesssim\|\varepsilon_{\mathbb{C}}^h(v_h)\|_{\mathbb{A}}.
\end{equation*}
\end{lemma}
The space $\mathcal{RM}_H$ can be viewed as a broken rotated Raviart--Thomas finite element space. Here we are interested in $Q_H$ instead of $P_H$ because we will use the fact $\mathcal{RM}\subset\ker(\varepsilon)$, see the proof of Lemma \ref{disapprox1} for details.

The next lemma is used to remove the Lam\'e coefficient $\lambda$ in error bounds. The proof can be found in Lemmas 3.1 and 3.2 of \cite{ArnoldDouglasGupta1984}. 
\begin{lemma}\label{rob}
There exists a constant $C_{\emph{rb}}$ depends only on $\mu$ and $\Omega$, such that
\begin{align*}
    \|\tau\|\leq C_{\emph{rb}}\big(\|\tau\|_{\mathbb{A}}+\|\divg\tau\|_{H^{-1}(\Omega)}\big)
\end{align*}
for all $\tau\in\Sigma$ with $\int_\Omega\tr\tau dx=0$.
\end{lemma}

For $w\in H^1(T)$ and $\phi\in H^1(T,\mathbb{R}^2)$,
we have the integration-by-parts formula:
\begin{align}\label{ip}
    (\curl w,\phi)_T=\langle w,\phi\cdot t\rangle_{\partial T}-(w,\rot\phi)_T.
\end{align}

With the above preparations, we are able to prove the discrete reliability of $\eta_h.$
\begin{theorem}[discrete reliability]\label{disupper}
Let $\Th, \CT_H\in\mathbb{T}$ with $\CT_H\leq\Th$. There exists a constant $C_{\emph{drel}}$ depending only on $\mu, \Omega$, $\gamma_0$, such that
\begin{equation*}
    \|\sigma_H-\sigma_h\|^2_A\leq C_{\emph{drel}}\big(\eta^2_H(\sigma_H,\widetilde{\CR}_H)+\osc^2_H(f,\CR_H)\big).
\end{equation*}
\end{theorem}
\begin{proof}
Applying Theorem \ref{disdecomp} to $\sigma_H-\sigma_h\in\Sigma_H$ gives
\begin{align}\label{sigmatotal}
    \sigma_H-\sigma_h=J(w_h)+\varepsilon_{\mathbb{C}}^h(v_h)
\end{align}
for some $w_h\in W_h$ and $v_h\in U_h$. Taking $\tau_H=\delta$ in \eqref{error1} leads to
\begin{align}\label{tr1}
    \int_\Omega\tr(\sigma_H-\sigma_h)dx=0.
\end{align}
Direct calculation shows that
\begin{equation}\label{tr2}
    \int_\Omega\frac{1}{2(\mu+\lambda)}\tr\varepsilon_{\mathbb{C}}^h(v_h) dx=(\mathbb{A}\varepsilon_{\mathbb{C}}^h(v_h),\delta)=-(v_h,\divg\delta)=0.
\end{equation}
Then a combination of \eqref{sigmatotal}--\eqref{tr2} yields
\begin{equation}\label{trace}
    \int_\Omega\tr J(w_h)dx=0.
\end{equation}
Hence using Lemma \ref{rob}, $\divg\circ J=0$, \eqref{trace} and the $\mathbb{A}$-orthogonality between $Jw_h$ and $\varepsilon_{\mathbb{C}}^h(v_h)$, we obtain the following robust bound
\begin{equation}\label{bdJw}
    \|Jw_h\|\lesssim\|Jw_h\|_{\mathbb{A}}\leq\|\sigma_H-\sigma_h\|_{\mathbb{A}}.
\end{equation}
Let $E_h=w_h-I_Hw_h.$  Using \eqref{error1}, $\divg\circ J=0$, \eqref{localT}, we have
\begin{equation}\label{total}
    \begin{aligned}
    &(\mathbb{A}(\sigma_H-\sigma_h),Jw_h)=(\mathbb{A}(\sigma_H-\sigma_h),JE_h)\\
    &=(\mathbb{A}\sigma_H,JE_h)+\langle\partial_t\curl E_h,g_D\rangle_{\Gamma_D}\\
    &=\sum_{T\in\widetilde{\mathcal{R}}_H}(\mathbb{A}\sigma_H,JE_h)_T-\langle\curl E_h,\partial_tg_D\rangle_{\Gamma_D}.
\end{aligned}
\end{equation}
In the last equality, we integrate by parts on $\Gamma_D$ and use the fact $\curl E_h=0$ on $\partial\Gamma_D\subset\overline{\Gamma}_N$ (see \eqref{GammaN}, \eqref{dnGammaN}).
For each $T\in\widetilde{\mathcal{R}}_H$, using the formula \eqref{ip}, we have
\begin{equation}\label{part1}
    \begin{aligned}
    &(\mathbb{A}\sigma_H,JE_h)_T=\langle (\mathbb{A}\sigma_H)t,\curl E_h\rangle_{\partial T}-(\rot \mathbb{A}\sigma_H,\curl E_h)_T\\
    &\quad=\langle(\mathbb{A}\sigma_H)t,\curl E_h\rangle_{\partial T}-\langle(\rot \mathbb{A}\sigma_H)\cdot t,E_h\rangle_{\partial T}+(\rot\rot \mathbb{A}\sigma_H,E_h)_T.
    \end{aligned}
\end{equation}
Integrating by parts on each edge of $T$, we have
\begin{equation}\label{part2}
\begin{aligned}
\langle(\mathbb{A}\sigma_H)t,\curl E_h\rangle_{\partial T}&=\langle t^\top (\mathbb{A}\sigma_H)t,\partial_nE_h\rangle_{\partial T}-\langle n^\top(\mathbb{A}\sigma_H)t,\partial_tE_h\rangle_{\partial T}\\
&=\langle t^\top (\mathbb{A}\sigma_H)t,\partial_nE_h\rangle_{\partial T}+\langle n^\top\partial_t(\mathbb{A}\sigma_H)t,E_h\rangle_{\partial T}.
\end{aligned}
\end{equation}
In the last equality, we use the property \eqref{localV}, i.e., $E_h=0$ at each vertex of $T$. Similarly, for the boundary term in \eqref{total}, 
\begin{equation}\label{boundaryterm}
\begin{aligned}
        &-\langle\curl E_h,\partial_tg_D\rangle_{\Gamma_D}=-\langle\partial_nE_h,t\cdot\partial_tg_D\rangle_{\Gamma_D}+\langle\partial_tE_h,n\cdot\partial_tg_D\rangle_{\Gamma_D}\\
        &\qquad=-\langle\partial_nE_h,t\cdot\partial_tg_D\rangle_{\Gamma_D}-\langle E_h,n\cdot\partial^2_tg_D\rangle_{\Gamma_D}.
\end{aligned}
\end{equation}

Let ${\CE}^o_H(\widetilde{\CR}_H)=\{e\in\mathcal{E}^o_H: e\subset\partial T\text{ for some }T\in\widetilde{\CR}_H\}$ and ${\CE}^D_H(\widetilde{\CR}_H)=\{e\in\mathcal{E}_H: e\subset\partial T\text{ for some }T\in\widetilde{\CR}_H, e\subset\Gamma_D\}$. Note that $E_h$ and $\partial_nE_h$ are continuous over each edge in $\CT_H.$ Combining \eqref{part1}--\eqref{boundaryterm} and using \eqref{localT}, \eqref{GammaN}, \eqref{dnGammaN}, we obtain
\begin{equation}\label{mainpart}
    \begin{aligned}
    &(\mathbb{A}(\sigma_H-\sigma_h),Jw_h)=\sum_{T\in\widetilde{\CR}_H}\left\{\langle t^\top (\mathbb{A}\sigma_H)t, \partial_nE_h\rangle_{\partial T}\right.\\
    &\qquad+\left.(\rot\rot \mathbb{A}\sigma_H,E_h)_T+\langle n^\top\partial_t(\mathbb{A}\sigma_H)t-(\rot \mathbb{A}\sigma_H)\cdot t,E_h\rangle_{\partial T}\right\}\\
    &\qquad\quad-\sum_{e\in\CE^D_H(\widetilde{\CR}_H)}\left\{\langle\partial_{n_e}E_h,t_e\cdot\partial_{t_e}g_D\rangle_e+\langle E_h,n_e\cdot\partial^2_{t_e}g_D\rangle_e\right\}\\
    &=\sum_{e\in\CE^o_H(\widetilde{\CR}_H)}\left\{\langle t_e^\top\lr{\mathbb{A}\sigma_H}t_e,\partial_{n_e}E_h\rangle_e+\langle n^\top_e\partial_{t_e}\lr{\mathbb{A}\sigma_H}t_e-\lr{\rot \mathbb{A}\sigma_H}\cdot t_e,E_h\rangle_e\right\}\\
    &\qquad+\sum_{T\in\widetilde{\CR}_H}(\rot\rot \mathbb{A}\sigma_H,E_h)_T+\sum_{e\in\CE^D_H(\widetilde{\CR}_H)}\big\{\langle t_e^\top(\mathbb{A}\sigma_Ht_e-\partial_{t_e}g_D),\partial_{n_e}E_h\rangle_e\\
    &\qquad\quad+\langle n^\top_e\partial_{t_e}(\mathbb{A}\sigma_H)t_e-(\rot \mathbb{A}\sigma_H)\cdot t_e-n_e^\top\partial_{t_e}^2g_D,E_h\rangle_e\big\}.
    \end{aligned}
\end{equation}
Using the expression \eqref{mainpart} and the Cauchy--Schwarz inequality, we have
\begin{equation}\label{Jpart}
    \begin{aligned}
    &(\mathbb{A}(\sigma_H-\sigma_h),Jw_h)\\
    &\quad\lesssim\eta_H(\sigma_H,\widetilde{\CR}_H)(\sum_{T\in\CT_H}h_T^{-4}\|E_h\|_T^2+h_T^{-1}\|\partial_nE_h\|_{\partial T}^2+h_T^{-3}\|E_h\|_{\partial T}^2)^\frac{1}{2}.
    \end{aligned}
\end{equation}
It then follows from \eqref{Jpart}, \eqref{approxG} and \eqref{bdJw} that
\begin{equation}\label{Jpartfinal}
    \begin{aligned}
    (\mathbb{A}(\sigma_H-\sigma_h),Jw_h)&\lesssim\eta_H(\sigma_H,\widetilde{\mathcal{R}}_H)|w_h|_{H^2(\Omega)}\\
    &\lesssim\eta_H(\sigma_H,\widetilde{\mathcal{R}}_H)\|\sigma_h-\sigma_H\|_{\mathbb{A}}.
    \end{aligned}
\end{equation}
On the other hand, \eqref{dismix2} implies
\begin{equation}\label{interEA}
\begin{aligned}
&(\mathbb{A}(\sigma_H-\sigma_h),\varepsilon_{\mathbb{C}}^h(v_h))=-(\divg(\sigma_H-\sigma_h),v_h)=(P_hf-P_Hf,v_h)\\
&\quad=(f-P_Hf,v_h-Q_Hv_h)=\sum_{T\in\mathcal{R}_H}(f-P_Hf,v_h-Q_Hv_h)_T.
\end{aligned}
\end{equation}
Using \eqref{interEA}, Lemma \ref{disapprox}, and $\|\varepsilon_{\mathbb{C}}^h(v_h)\|_{\mathbb{A}}\leq\|\sigma_H-\sigma_h\|_{\mathbb{A}},$  we obtain
\begin{equation}\label{epart}
\begin{aligned}
&(\mathbb{A}(\sigma_H-\sigma_h),\varepsilon_{\mathbb{C}}^h(v_h))\leq\osc_H(f,\mathcal{R}_H)\big(\sum_{T\in\CR_H}h_T^{-2}\|v_h-Q_Hv_h\|_T^2\big)^\frac{1}{2}\\
    &\qquad\lesssim\osc_H(f,\mathcal{R}_H)\|\varepsilon_{\mathbb{C}}^h(v_h)\|_{\mathbb{A}}\leq\osc_H(f,\mathcal{R}_H)\|\sigma_H-\sigma_h\|_{\mathbb{A}}.
\end{aligned}
\end{equation}
Finally, a combination of \eqref{Jpartfinal} and \eqref{epart} completes the proof.
\end{proof}

Let $\CT_h$ be a uniform refinement of $\CT_H$ and let the maximum mesh size of $\Th$ go to $0$ in Theorem \ref{disupper}. In this case, $\widetilde{\CR}_H=\CR_H=\CT_H$ and $\sigma_h\rightarrow\sigma, u_h\rightarrow u$ in $\Sigma\times U$. Therefore, we obtain the continuous upper bound
\begin{align}\label{rel}
\|\sigma-\sigma_H\|_{\mathbb{A}}^2&\leq{C}_{\text{rel}}\big({\eta}^2_H(\sigma_H)+\osc^2_H(f)\big),
\end{align}
where ${C}_{\text{rel}}\in(0,C_{\text{drel}}]$ is a constant depending only on $\mu, \Omega, \gamma_0$. 

The quasi-orthogonality on the variable $\sigma$ follows with a similar argument as in the proof of Theorem \ref{disupper}.
\begin{theorem}[quasi-orthogonality]\label{qosigma}
Let $\Th, \CT_H\in\mathbb{T}$ with $\CT_H\leq\Th.$ For any $\nu\in(0,1)$, it holds that
\begin{align*}
    (1-\nu)\|\sigma-\sigma_h\|^2_\mathbb{A}\leq\|\sigma-\sigma_H\|^2_\mathbb{A}-\|\sigma_h-\sigma_H\|^2_\mathbb{A}+C_{\nu}\osc^2_H(f,\mathcal{R}_H),
\end{align*}
where $C_\nu=\nu^{-1}C_\sigma$ and $C_\sigma$ is a constant depending only on $\mu, \Omega, \gamma_0$.
\end{theorem}
\begin{proof}
Combining \eqref{sigmatotal} and \eqref{mix1}, \eqref{dismix1} yields
\begin{equation}\label{qopart1}
    (\mathbb{A}(\sigma-\sigma_h),\sigma_H-\sigma_h)=(\mathbb{A}(\sigma-\sigma_h),\varepsilon_{\mathbb{C}}^h(v_h))\leq\|\sigma-\sigma_h\|_\mathbb{A}\|\varepsilon_{\mathbb{C}}^h(v_h)\|_\mathbb{A}.
\end{equation}
Following the same analysis in \eqref{interEA}, we have
\begin{equation}\label{qopart2}
    \begin{aligned}
    &\|\varepsilon_{\mathbb{C}}^h(v_h)\|_\mathbb{A}^2=-(\divg\varepsilon_\mathbb{C}^h(v_h),v_h)=-(\divg(\sigma_H-\sigma_h),v_h)\\
    &\quad=\sum_{T\in\mathcal{R}_H}(f-P_Hf,v_h-Q_Hv_h)_T\leq C_\sigma^\frac{1}{2}\osc_H(f,\mathcal{R}_H)\|\varepsilon_\mathbb{C}^h(v_h)\|_\mathbb{A}.
\end{aligned}
\end{equation}
A combination of \eqref{qopart1} and \eqref{qopart2} shows that
\begin{align*}
    &(\mathbb{A}(\sigma-\sigma_h),\sigma_H-\sigma_h)\leq C_\sigma^\frac{1}{2}\|\sigma-\sigma_h\|_\mathbb{A}\osc_H(f,\mathcal{R}_H)\\
    &\qquad\leq\frac{\nu}{2}\|\sigma-\sigma_h\|^2_\mathbb{A}+\frac{\nu^{-1}}{2}C_\sigma\osc^2_H(f,\mathcal{R}_H),
\end{align*}
where $0<\nu<1$. Therefore
\begin{align*}
    &\|\sigma-\sigma_h\|^2_\mathbb{A}=\|\sigma-\sigma_H\|^2_\mathbb{A}-\|\sigma_h-\sigma_H\|^2_\mathbb{A}+2(\mathbb{A}(\sigma-\sigma_h),\sigma_H-\sigma_h)\\
    &\quad\leq\|\sigma-\sigma_H\|^2_\mathbb{A}-\|\sigma_h-\sigma_H\|^2_\mathbb{A}+\nu\|\sigma-\sigma_h\|^2_\mathbb{A}+\nu^{-1}C_\sigma\osc^2_H(f,\mathcal{R}_H).
\end{align*}
The proof is complete.
\end{proof}

\section{Quasi-optimality}\label{secqo}
Define
\begin{align*}
    \bar{\eta}_h(\sigma_h,T)&=\big(\eta^2_h(\sigma_h,T)+\osc^2_h(f,T)\big)^\frac{1}{2},\\
    \bar{\eta}_h(\sigma_h,\mathcal{M})&=\big(\sum_{T\in\mathcal{M}}\bar{\eta}^2_h(\sigma_h,T)\big)^\frac{1}{2},\quad\mathcal{M}\subseteq\Th.
\end{align*}
Our adaptive algorithm is based on the classical feedback loop \begin{equation*}
\begin{CD}
        \textsf{SOLVE}@>>>\textsf{ESTIMATE} @>>>\textsf{MARK}@>>>\textsf{REFINE}.
\end{CD}
\end{equation*}
In the procedure \textsf{REFINE}, we use the newest vertex bisection \cite{Mitchell1990,Ban1991,Stevenson2008} to ensure the shape regularity of $\mathbb{T}$.

\begin{algorithm}\label{AMFEM}
Input the initial mesh $\mathcal{T}_{h_0}=\mathcal{T}_0$ and $\theta\in(0,1)$. Set $\ell=0$. 
\begin{itemize}
\item[]\textsf{SOLVE}: Solve \eqref{dismix} on $\mathcal{T}_{h_\ell}$ to obtain the finite element solution $(\sigma_{h_\ell},u_{h_\ell})$. 
\item[]\textsf{ESTIMATE}: Calculate error indicators $\{\bar{\eta}_{h_\ell}(\sigma_{h_\ell},T)\}_{T\in\CT_{h_\ell}}$.
\item[]\textsf{MARK}: Select a subset $\mathcal{M}_{\ell}\subset\mathcal{T}_{h_\ell}$ with minimal cardinality such that  $$\bar{\eta}_{h_\ell}(\sigma_{h_\ell},\mathcal{M}_{\ell})\geq\theta\bar{\eta}_{h_\ell}.$$
\item[]\textsf{REFINE}: Refine all elements in $\mathcal{M}_{\ell}$ and minimal number of neighboring elements to remove hanging nodes. The resulting conforming mesh is $\mathcal{T}_{h_{\ell+1}}$. Set $\ell=\ell+1$. Go to \textsf{SOLVE}.
\end{itemize}
\end{algorithm}

In the procedure \textsf{ESTIMATE}, the actual estimator is $\big({\eta}^2_{h_\ell}+{\osc}^2_{h_\ell}\big)^\frac{1}{2}$ instead of $\eta_{h_\ell}$. Due to this strategy, an extra marking step  for data oscillation can be avoided, see, e.g., \cite{HuangXu2012,HLMS2019}. Since the data oscillation $\osc_{h_\ell}$ is completely local, its behavior can be easily described by the following lemma, see, e.g., Lemma 5.2 in \cite{HLMS2019}.
\begin{lemma}\label{osc}
For $\ell\geq0,$ let $\mathcal{R}_\ell=\CT_{h_\ell}\backslash\CT_{h_{\ell+1}}$ denote the collection of refinement elements from $\CT_{h_\ell}$ to $\CT_{h_{\ell+1}}$. It holds that 
\begin{align*}
    &\osc_{h_{\ell+1}}^2\leq\osc^2_{h_\ell}-\frac{1}{2}\osc^2_{h_\ell}(f,\mathcal{R}_\ell).
\end{align*}
\end{lemma}

Once the theoretical results in Section \ref{secsigma} are available, the convergence and complexity analysis of Algorithm \ref{AMFEM} follows from classical and systematic arguments, see, e.g., \cite{CKNS2008,Stevenson2007} and \cite{CarstensenFeischlPagePraetorius2014} for axioms of adaptivity. To be self-contained, we still briefly outlined the rest of adaptivity analysis.

The estimator reduction is a standard ingredient in the convergence analysis of AFEMs, see \cite{CKNS2008}. Since $\bar{\eta}_h$ involves data oscillation, we refer to Lemma 5.1 in \cite{HLMS2019} for a detailed proof. The only new ingredient in the proof of Lemma \ref{reduction} is the robust inequality $\|\mathbb{A}\tau\|\lesssim\|\tau\|_\mathbb{A}.$
\begin{lemma}\label{reduction}
There exists a constant $\gamma\in(0,1)$ and $C_{\emph{re}}>0$ depending only on $\mu, \Omega, \gamma_0,\theta$ such that
\begin{align*}
    \bar{\eta}_{h_{\ell+1}}^2\leq\gamma\bar{\eta}_{h_\ell}^2+C_{\emph{re}}\|\sigma_{h_\ell}-\sigma_{h_{\ell+1}}\|_\mathbb{A}^2.
\end{align*}
\end{lemma}

For convenience, let
\begin{align*}
    e_{\ell}=\|\sigma-\sigma_{h_\ell}\|_\mathbb{A},\quad E_{\ell}=\|\sigma_{h_\ell}-\sigma_{h_{\ell+1}}\|_\mathbb{A}.
\end{align*}
The next theorem gives the contraction property of Algorithm \ref{AMFEM}.
\begin{theorem}[contraction]
There exists constants $\nu, \alpha\in(0,1)$ depending only on $\theta, \mu, \Omega, \gamma_0$ such that 
\begin{align*}
    (1-\nu)e^2_{\ell+1}+2C_\nu\osc^2_{h_{\ell+1}}+C^{-1}_{\emph{re}}\bar{\eta}_{h_{\ell+1}}^2\leq\alpha\big((1-\nu)e^2_{\ell}+2C_\nu\osc^2_{h_\ell}+C^{-1}_{\emph{re}}\bar{\eta}_{h_{\ell}}^2\big).
\end{align*}
\end{theorem}
\begin{proof}
A combination of Theorem \ref{qosigma} and Lemma \ref{osc} shows that
\begin{align}\label{ing1}
    (1-\nu)e_{\ell+1}^2+2C_\nu\osc_{h_{\ell+1}}^2\leq e_{\ell}^2+2C_\nu\osc^2_{h_{\ell}}-E_{h_\ell}^2.
\end{align}
On the other hand, the reliability \eqref{rel} gives
\begin{align}\label{ing2}
    e^2_\ell\leq C_{\text{rel}}\bar{\eta}_{h_\ell}^2,\quad \osc^2_{h_\ell}\leq\bar{\eta}^2_{h_\ell}.
\end{align}
Let $\alpha\in(0,1)$ be a constant.
Using \eqref{ing1} and Lemma \ref{reduction}, we have 
\begin{equation*}
    \begin{aligned}
    &(1-\nu)e_{\ell+1}^2+2C_\nu\osc_{h_{\ell+1}}^2+C_{\text{re}}^{-1}\bar{\eta}_{h_{\ell+1}}^2\leq e_{\ell}^2+2C_\nu\osc_{h_{\ell}}^2+C_{\text{re}}^{-1}\gamma\bar{\eta}_{h_\ell}^2\\
    &\quad\leq\alpha(1-\nu)e_{\ell}^2+2\alpha C_\nu\osc_{h_{\ell}}^2+\big(1-\alpha(1-\nu)\big)e_\ell^2+2(1-\alpha) C_\nu\osc_{h_\ell}^2+C_{\text{re}}^{-1}\gamma\bar{\eta}_{h_\ell}^2.
    \end{aligned}
\end{equation*}
Combining it with \eqref{ing2} yields
\begin{equation}\label{con1}
    \begin{aligned}
    &(1-\nu)e_{\ell+1}^2+2C_\nu\osc_{h_{\ell+1}}^2+C_{\text{re}}^{-1}\bar{\eta}_{h_{\ell+1}}^2\leq\alpha(1-\nu)e_{\ell}^2+2\alpha C_\nu\osc_{h_\ell}^2\\
    &\quad+\{\big(1-\alpha(1-\nu)\big)C_{\text{rel}}+2(1-\alpha)C_\nu+C_{\text{re}}^{-1}\gamma\}\bar{\eta}_{h_\ell}^2.
    \end{aligned}
\end{equation}
Let $\big(1-\alpha(1-\nu)\big)C_{\text{rel}}+2(1-\alpha)C_\nu+C^{-1}_{\text{re}}\gamma=\alpha C^{-1}_{\text{re}},$ i.e., 
\begin{align*}
    \alpha=\frac{C_{\text{rel}}+2C_\nu+C^{-1}_{\text{re}}\gamma}{(1-\nu)C_{\text{rel}}+2C_\nu+C^{-1}_{\text{re}}},\quad
\end{align*}
Clearly $\alpha<1$ provided $0<\nu<\frac{1-\gamma}{C_{\text{re}}C_{\text{rel}}}$. The contraction follows from \eqref{con1}.
\end{proof}

The efficiency of $\bar{\eta}_h$ follows with the same bubble function technique in \cite{CHHM2018,CGG2019}:
\begin{align}\label{eff}
    C_{\text{eff}}\bar{\eta}^2_h(\sigma_h)&\leq\|\sigma-\sigma_h\|^2_\mathbb{A}+\osc_h^2(f),
\end{align}
where the constant $C_{\text{eff}}>0$ depends only on $\mu, \Omega, \gamma_0$, $r$. 
An essential ingredient in the complexity analysis is the following cardinality estimate
\begin{equation}\label{card}
    \#\mathcal{T}_{h_\ell}-\#\mathcal{T}_{h_0}\lesssim\sum_{j=0}^{\ell-1}\mathcal{M}_j.
\end{equation}
It has been shown in \cite{BinevDahmenDeVore2004} that \eqref{card} holds provided the newest vertices in the initial mesh $\CT_{h_0}$ are suitably chosen.
In addition, the marking parameter $\theta$ is required to be below the threshold $$\theta_{*}=\min\left(1,\frac{C_{\text{eff}}}{3C_{\text{drel}}}\right)^\frac{1}{2},$$
which can be derived in the next lemma.
\begin{lemma}\label{optmarking}
\emph{(optimal marking)}
Let $\mathcal{T}_h, \CT_H\in\mathbb{T}$ with $\CT_H\leq\Th$. Set $\mu=\frac{1}{2}(1-\frac{\theta^2}{\theta^2_*})$. If 
\begin{equation}\label{testreduction}
\|\sigma-\sigma_h\|_\mathbb{A}^2+\osc^2_h(f)\leq\mu\left\{\|\sigma-\sigma_H\|_\mathbb{A}^2+\osc^2_H(f)\right\}.
\end{equation}
Then the set $\widetilde{\mathcal{R}}_H$ in Lemma \ref{propIH} verifies the D\"orfler marking property
$$\bar{\eta}_H(\sigma_H,\widetilde{\mathcal{R}}_H)\geq\theta\bar{\eta}_H.$$
\end{lemma} 
\begin{proof}
Using \eqref{eff} and \eqref{testreduction}, we have
\begin{equation}\label{interopt}
\begin{aligned}
&(1-2\mu)C_\text{eff}\bar{\eta}^2_H\leq(1-2\mu)\big(\|\sigma-\sigma_H\|_\mathbb{A}^2+\osc^2_H(f)\big)\\
&\qquad\leq\|\sigma-\sigma_H\|_\mathbb{A}^2-2\|\sigma-\sigma_h\|_\mathbb{A}^2+\osc^2_H(f)-2\osc^2_h(f)\\
&\qquad\leq2\|\sigma_H-\sigma_h\|_\mathbb{A}^2+\osc_H^2(f,\CR_H).
\end{aligned}
\end{equation}
In the last step, we use the obvious inequality $$\osc^2_H(f)-2\osc^2_h(f)\leq\osc^2_H(f)-\osc^2_h(f)\leq\osc_H^2(f,\CR_H).$$ It then follows from \eqref{interopt} and Theorem \ref{disupper} that
$$(1-2\mu)C_\text{eff}\bar{\eta}^2_H\leq3C_{\text{drel}}\bar{\eta}^2_H(\sigma_H,\widetilde{\CR}_H).$$
The proof is then complete by $\theta^2_*\leq\frac{C_\text{eff}}{3C_\text{drel}}$.
\end{proof}
Under these assumptions, the convergence rate of can be characterized by the nonlinear approximation property of $\sigma$ and $f$. Let $\widetilde{\mathbb{T}}$ be the collection of grids created by newest vertex bisection from $\mathcal{T}_{h_0}.$ For $s>0$, define the semi-norms
\begin{align*}
    |\sigma|_{s}&=\sup_{N>0}\{N^s\min_{\CT_h\in\widetilde{\mathbb{T}},\#\Th-\#\CT_{h_0}\leq N}\min_{\tau_h\in \Sigma_h}\|\sigma-\tau_h\|_\mathbb{A}\},\\
        |f|^o_{s}&=\sup_{N>0}\{N^s\min_{\CT_h\in\widetilde{\mathbb{T}},\#\Th-\#\CT_{h_0}\leq N}\osc_h(f)\}.
\end{align*}
One can also define the coupled approximation semi-norm
\begin{align*}
    |(\sigma,f)|_s:=\sup_{N>0}\{N^s\min_{\Th\in\widetilde{\mathbb{T}},\#\CT_h-\#\CT_{h_0}\leq N}(\|\sigma-\sigma_h\|_\mathbb{A}^2+\osc_h(f)^2)^\frac{1}{2}\}.
\end{align*}
Since $\lambda, \mu$ are constants, we have the following equivalence 
\begin{align*}
    |(\sigma,f)|_s<\infty\Leftrightarrow|\sigma|_s+|f|_s^o<\infty
\end{align*}
as argued in Lemma 5.3 of \cite{CKNS2008}.
The quasi-optimal convergence rate of Algorithm \ref{AMFEM} follows from the contraction and previous assumptions, see, e.g., \cite{CKNS2008}, Lemma 5.10 and Theorems 5.1 for details.
\begin{theorem}[quasi-optimality]\label{qopt}
Let $\{(\sigma_{h_\ell},u_{h_\ell},\CT_{h_\ell})\}_{\ell\geq0}$ be a sequence of finite element solutions and meshes generated by Algorithm \ref{AMFEM}.  Assume $|\sigma|_s+|f|_s^o<\infty$,  $\theta\in(0,\theta_{*})$, and \eqref{card} hold. There exists a constant $C_{\emph{opt}}$ depending only on $\theta, \theta_*, \alpha$, $\mu, \Omega$, $\gamma_0$, such that
\begin{equation*}
\big(\|\sigma-\sigma_{h_\ell}\|_\mathbb{A}^2+\osc^2_{h_\ell}\big)^\frac{1}{2}\leq C_{\emph{opt}}\big(|\sigma|_{s}+|f|_s^o\big)\big(
\#\CT_{h_\ell}-\#\CT_{h_0}\big)^{-s}.
\end{equation*}
\end{theorem}

\section{Local interpolation and discrete approximation}\label{secdisapprox}
In this section, we give proofs of Proposition \ref{propIH} and Lemma \ref{disapprox}. The $H^1$-bounded regularized interpolation in \cite{GiraultScott2002} does not satisfy the property \eqref{localV}. For our purpose, an $H^2$-bounded interpolation is enough and we will not regularize the degrees of freedom based on point evaluation.
\begin{proof}[Proof of Proposition \ref{propIH}]
Given $x\in\mathcal{V}_H$, $e\in\CE_H$, $w_H\in W_H$, we say $\partial_e\partial_ew_H(x)=\partial^2_e w_H(x)$ is a second edge derivative at $x,$ where $\partial_e=\partial_{t_e}$ is the directional derivative along $t_e$. For two edges $e, e^\prime$ of $T\in\CT_H$ having $x$, $\partial_e\partial_{e^\prime}(w_H|_T)(x)$ is called a cross derivative at $x.$ A vertex $x$ is called  singular if all edges in $\mathcal{E}_H$ meeting at $x$ fall on two straight lines.
The nodal variables (global degrees of freedom) of $w\in W_H$ given in \cite{GiraultScott2002} are briefly described as follows.
\begin{enumerate}
    \item the value of $w$ and $\nabla w$ at each vertex $x\in\mathcal{V}_H$;\label{dof1}
    \item the edge normal derivative $\partial_{n_e}w$ at $r+1$ distinct interior points of each edge $e\in\CE_H$;
    \item the value of $w$ at $r$ distinct interior points of each edge $e\in\CE_H$;
    \item the value of $w$ at $r(r+1)/2$ distinct interior points of each triangle $T\in\CT_H$;
    \item one cross derivative of $w$ at each vertex $x\in\mathcal{V}_H$ and two cross derivatives of $w$ at each nonsingular boundary vertex $x\in\mathcal{V}_H$;
    \item the second edge derivatives of $w$ for all edges meeting at each vertex $x\in\mathcal{V}_H,$ 
with the exception of one interior edge per nonsingular vertex.
\end{enumerate}
To preserve boundary conditions, for each boundary vertex, one edge used for defining edge and cross derivatives at that vertex is chosen to be on $\partial\Omega.$

Let $\{w\rightarrow D_iw(a_i)\}_{i=1}^N$ denote the collection of the above nodal variables, where a node $a_i$ could be a vertex in $\CT_H$ or an interior point of an edge/triangle in $\CT_H$, $D_i$ is a differential operator of order $|D_i|=0$ (point evaluation), $1$ (edge normal derivative), or $2$ (second edge derivative and cross derivative). Note that  $\{a_i\}_{i=1}^N$ are not distinct since a node may be associated with multiple differential operators. 

If $a_i\in\mathcal{V}_H$ and $D_i=\partial_{e^1_i}\partial_{e^2_i}$ is a second edge derivative or cross derivative, let $e_i=e^1_i\ni a_i$. If $a_i\in\mathcal{V}_H$ and $D_i=\partial_{x_1}$ or $\partial_{x_2}$, let $e_i\ni a_i$ be any edge in $\CE_H$. If $a_i$ is an interior point of $e\in\CE_H$, we  choose $e_i=e$. Moreover, $e_i$ is chosen to be in $\partial\Omega$ if $a_i$ is a boundary vertex. By Riesz's representation theorem, there exists a polynomial $\psi_i\in\mathcal{P}_{r+5}(e_i)$, such that
\begin{align}\label{Riesz}
    \int_{e_i}w\psi_ib_ids=w(a_i)\text{ for all }w\in\mathcal{P}_{r+5}(e_i),
\end{align}
where $b_i$ is the edge bubble polynomial of unit size vanishing on the boundary of  $e_i$. For a second edge derivative or cross derivative $\partial_{e^1_i}\partial_{e^2_i}$, using \eqref{Riesz} and integrating by parts formula on $e_i$, we have
\begin{equation}\label{Riesz2}
    \begin{aligned}
    &(\partial_{e^1_i}\partial_{e^2_i}w)(a_i)=\int_{e_i}\partial_{e^1_i}\partial_{e_i^2}w(a_i)\psi_ib_ids\\
    &\quad=-\int_{e_i}\partial_{e_i^2}w(a_i)\partial_{e_i^1}(\psi_ib_i)ds\text{ for all }w\in\mathcal{P}_{r+5}(e_i).
    \end{aligned}
\end{equation}

Let $\{\phi_i\}_{i=1}^N$ be the basis dual to the unisolvent set $\{w\rightarrow D_iw(a_i)\}_{i=1}^N$. For $w_h\in W_h$, we define the
interpolant $I_Hw_h$ as
\begin{equation}\label{IHwh}
    \begin{aligned}
    I_Hw_h&=\sum_{|D_i|=0}w_h(a_i)\phi_i+\sum_{|D_i|=1}\left(\int_{e_i}({D}_iw_h){\psi}_ib_ids\right)\phi_i\\
    &\quad-\sum_{|D_i|=2}\left(\int_{e_i}\partial_{e_i^2}w_h(a_i)\partial_{e^1_i}(\psi_ib_i)ds\right)\phi_i.
\end{aligned}
\end{equation}
The definition of $I_H$ implies $w_h-I_Hw_h=0$ at each node $a_i$ and in particular \eqref{localV}.
For $T\in\CT_H\backslash\widetilde{\CR}_H$, the choice of $e_i$ implies $w_h|_{e_i}\in\mathcal{P}_{r+5}(e_i), \nabla w_h|_{e_i}\in\mathcal{P}_{r+4}(e_i)$. Using this fact and \eqref{Riesz}, \eqref{Riesz2}, we obtain 
\begin{align}\label{localDi}
    D_i(I_Hw_h)=D_iw_h\text{ for all }a_i\subset \overline{T}.
\end{align}
Due to \eqref{localDi} and the unisolvence of $\{D_i\}_{a_i\in T}$ (see the analysis in \cite{MS1975,GiraultScott2002}), we have $(w_h-I_Hw_h)|_T=0$ and thus verify the property \eqref{localT}. 

We note that $w_h\in W_h$ implies $(\nabla w_h)|_{\Gamma_j}$ is constant and thus $D_iw_h=D_iI_Hw_h$ with $|D_i|=1$ for each boundary node $a_i\in\Gamma_j$. It could also be observed from \eqref{IHwh}, \eqref{Riesz2}, and constancy of $(\nabla w_h)|_{\Gamma_j}$ that $\partial_{e_i}\partial_{e_i}I_Hw_h=\partial_{e_i}\partial_{e_i}w_h$ for each boundary $e_i\subset\Gamma_j$. Therefore enough nodal variables vanish to enforce $w_h=I_Hw_h$ on $\Gamma_j$ and \eqref{GammaN} is confirmed. 

For the same reason above, $\partial_{e^1_i}\partial_{e^2_i}w_h=\partial_{e^1_i}\partial_{e^2_i}I_Hw_h$ for each cross derivative assigned to boundary node $a_i\in\Gamma_j$. Therefore enough nodal variables vanish to enforce $\partial_n(w_h-I_Hw_h)|_{\Gamma_j}=0$.

The interpolation estimate \eqref{approxG} directly follows from the same proof of Theorem 7.3 in \cite{GiraultScott2002} together with a trace inequality.
\end{proof}

The rest of this section is devoted to the proof of Lemma \ref{disapprox}.
First we present a modified  Korn's inequality on each local triangle $T\in\Th$.
\begin{lemma}\label{Korn1}
Given $\Th\in\mathbb{T}$, $T\in\Th$ and $v\in H^1(T,\mathbb{R}^2),$ we have
\begin{align*}
    |v|_{H^1(T)}\leq {C}_{\emph{Korn}}(\|\varepsilon(v)\|_{T}+h_T^{-1}\|Q_Tv\|_T),
\end{align*}
where $Q_Tv$ is the $L^2$ projection of $v$ onto $\mathcal{RM}$, and $C_{\emph{Korn}}$ is a constant relying only on $\gamma_0$. 
\end{lemma}
\begin{proof}
The standard compactness argument (cf.~Theorem 11.2.16 of \cite{BrennerScott2008}) implies 
\begin{align}\label{korn0}
    \|v\|_{H^1(T)}\leq{C}_T(\|\varepsilon(v)\|_{T}+\|Q_Tv\|_T),
\end{align}
where ${C}_T$ is a constant depending on $T.$ It remains to estimate $C_T$ by a homogeneity argument. Consider a reference triangle ${K}$ and the affine mapping $F_T: {K}\rightarrow T$ given by
$F_T(x)=B_Tx+b_T.$
Define 
\begin{align*}
    \Phi(B_T)=\sup_{v\in H^1(K,\mathbb{R}^2),\|v\|_{H^1({K})}=1}\Phi_{v}(B_T),
\end{align*}
where
\begin{align*}
    \Phi_{v}(B_T)&=\frac{|v\circ F_T^{-1}|_{H^1(T)}}{\|\varepsilon(v\circ F_T^{-1})\|_T+h_T^{-1}\|Q_T(v\circ F_T^{-1})\|_T}.
\end{align*}
Note that $\Phi_v$ is independent of $b_T.$
Due to \eqref{korn0}, the function $\Phi$ is well-defined. It is straightforward to check that $$\{\Phi_v(\cdot)\}_{v\in H^1({K},\mathbb{R}^2),\|v\|_{H^1({K})}=1}$$ is a family of equicontinuous functions on $GL(2,\mathbb{R}).$ Hence $\Phi$ defined by taking the supremum of this family must be continuous on $GL(2,\mathbb{R}).$ Let $\widehat{T}=\{h_T^{-1}x: x\in{T}\}$ be the scaled triangle of unit size. Since $\Th$ is shape regular, $\{B_{\widehat{T}}: T\in\Th\}$ is contained in a  compact subset of $GL(2,\mathbb{R})$, see, e.g., \cite{BrennerScott2008}. Combining the continuity and compactness, we obtain
\begin{align}\label{bdhat}
    \sup_{T\in\Th}\Phi(B_{\widehat{T}})=C_{\text{sup}}<\infty,
\end{align}
where $C_{\text{sup}}$ depends on the shape regularity of $\Th$. Therefore using a \emph{scaling} transformation and \eqref{bdhat}, we obtain
\begin{align*}
    \Phi(B_T)&=\sup_{v\in H^1({K},\mathbb{R}^2),\|v\|_{H^1({K})}=1}\frac{|v\circ F_{\widehat{T}}^{-1}|_{H^1(\widehat{T})}}{\|\varepsilon(v\circ F_{\widehat{T}}^{-1})\|_{\widehat{T}}+\|Q_{\widehat{T}}({v}\circ F_{\widehat{T}}^{-1})\|_{\widehat{T}}}\\
    &\lesssim\Phi(B_{\widehat{T}})\leq C_{\text{sup}}.
\end{align*}
The proof is complete.
\end{proof}
Then we present a modified discrete Korn's inequality on a triangle $T$. 
\begin{lemma}\label{disapprox1}
Let $\Th, \CT_H\in\mathbb{T}$ with $\CT_H\leq\Th$. For any $T\in\mathcal{T}_H$, let $\Th(T)=\{T^\prime\in\Th: T^\prime\subset T\}$ and $\CE_h(\mathring{T})=\{e\in\CE_h: e\subset T, e\not\subseteq\partial T\}$. Then for $v_h\in U_h|_T,$ we have
\begin{equation*} 
h_T^{-2}\|v_h-Q_Tv_h\|^2_{T}\lesssim\sum_{T^\prime\in\Th(T)}\|\varepsilon(v_h)\|_{T^\prime}^2+\sum_{e\in\CE_h(\mathring{T})}h_e^{-1}\|\lr{v_h}\|_e^2.
\end{equation*}
\end{lemma}
\begin{proof}
Let $w=v_h-Q_Tv_h$ and  $|w|_{H^1_h(T)}^2:=\sum_{T^\prime\in\Th(T)}|w|^2_{H^1(T^\prime)}$. Let \begin{equation*}
    S_h(T)=\{\tilde{v}\in C^0(T): \tilde{v}|_{T^\prime}\in\mathcal{P}_{r+2}(T^\prime)\text{ for }T^\prime\in\Th(T)\}
\end{equation*}
be the usual Lagrange element space of degree $r+2$. Following the analysis in \cite{Brenner2003,BrennerScott2008,HuangXu2012}, we construct a continuous piecewise polynomial function $Ew\in S_h(T)$ by setting the nodal value as
\begin{align*}
     Ew(x)=\frac{1}{\#\omega_{h,x}}\sum_{T^\prime\in\omega_{h,x}}(w|_{T^\prime})(x),
\end{align*}
where $x$ is a Lagrange node for the space $S_h(T)$ and $\omega_{h,x}=\{T^\prime\in\Th(T): x\in T^\prime\}$. An elementary estimate shows that
\begin{align}\label{errorE}
    h_T^{-2}\|w-Ew\|^2_T+|w-Ew|^2_{H^1_h(T)}\lesssim \sum_{e\in{\CE}_h(\mathring{T})}h_{e}^{-1}\|\lr{w}\|_e^2,
\end{align}
see, e.g., the proof of Lemma 10.6.6 in \cite{BrennerScott2008} and Lemma 2.8 in \cite{HuangXu2012}.
For the continuous function $Ew$, the Poincar\'e inequality implies
\begin{align}\label{Poin}
    \|Ew-Q_T{Ew}\|_T&\lesssim h_T|Ew|_{H^1(T)},
\end{align}
Using \eqref{errorE}, \eqref{Poin}, the triangle inequality, and $Q_Tw=0,$ we have
\begin{equation}\label{disPoin}
    \begin{aligned}
    \|w\|^2_T&\lesssim\|w-Ew\|_T^2+\|Q_T({w-Ew})\|_T^2+\|Ew-Q_T{Ew}\|_T^2\\
    &\lesssim h_T^2\sum_{e\in{\CE}_h(\mathring{T})}h_{e}^{-1}\|\lr{w}\|_e^2+h_T^2|Ew|^2_{H^1(T)}.
\end{aligned}
\end{equation}
It remains to estimate $|Ew|_{H^1(T)}.$ 
Lemma \ref{Korn1} implies
\begin{align}\label{Korn2}
    |Ew|_{H^1(T)}\lesssim\|\varepsilon(Ew)\|_{T}+h_T^{-1}\|Q_TEw\|_T.
\end{align}
It then follows from the triangle inequality, \eqref{Korn2}, $\varepsilon(Q_Tv_h)=0$, and $Q_T(w)=0$  that
\begin{equation}\label{disKorn}
    \begin{aligned}
    |Ew|_{H^1(T)}^2&\lesssim\sum_{T^\prime\in\Th(T)}\big(\|\varepsilon(w)\|_{T^\prime}^2+\|\varepsilon(w-Ew)\|_{T^\prime}^2\big)+h_T^{-2}\|Q_T(Ew-w)\|_T^2\\
    &\leq\sum_{T^\prime\in\Th(T)}\|\varepsilon(v_h)\|_{T^\prime}^2+|w-Ew|_{H^1_h(T)}^2+h_T^{-2}\|w-Ew\|_{T}^2.
\end{aligned}
\end{equation}
Combining \eqref{disPoin}, \eqref{disKorn}, \eqref{errorE}, and $\lr{w}_e=\lr{v_h}_e$ completes the proof.
\end{proof}

For $v_h\in U_h$, define the mesh-dependent norm 
\begin{equation*} 
|v_h|_{1,h}:=\big(\sum_{T\in\Th}\|\varepsilon(v_h)\|_T^2+\sum_{e\in\CE_h}h_e^{-1}\|\lr{v_h}\|_e^2\big)^\frac{1}{2}.
\end{equation*} 
It has been shown in \cite{CHH2018b} that the following discrete inf-sup condition holds:
\begin{align}\label{infsup1}
    |v_h|_{1,h}\lesssim\sup_{\tau_h\in\Sigma_h^{\text{HZ}}}\frac{(\divg\tau_h,v_h)}{\|\tau_h\|}\text{ for all }v_h\in U_h.
\end{align}
With the above preparation, we are able to prove Lemma \ref{disapprox}.
\begin{proof}[Proof of Lemma \ref{disapprox}]
Using the inf-sup condition \eqref{infsup1} and the inclusion $\Sigma_h^{\text{HZ}}\subset\Sigma_h$, we obtain
\begin{align}\label{infsup2}
    |v_h|_{1,h}\lesssim\sup_{\tau_h\in\Sigma_h}\frac{(\divg\tau_h,v_h)}{\|\tau_h\|}=\sup_{\tau_h\in\Sigma_h}\frac{(\mathbb{A}\tau_h,\varepsilon_\mathbb{C}^h(v_h))}{\|\tau_h\|}.
\end{align}
It then follows from \eqref{infsup2} and $\|\tau_h\|_\mathbb{A}\lesssim\|\tau_h\|$ that
\begin{align}
    |v_h|_{1,h}\lesssim\|\varepsilon_\mathbb{C}^h(v_h)\|_\mathbb{A}.
\end{align}
Combining it with Lemma \ref{disapprox1}, we have
\begin{equation*}
\begin{aligned}
\sum_{T\in\CT_H}h_T^{-2}\|v_h-Q_H v_h\|_T^2&\lesssim\sum_{T\in\CT_H}\big(\sum_{T^\prime\in\Th(T)}\|\varepsilon(v_h)\|_{T^\prime}^2+\sum_{e\in\CE_h(\mathring{T})}h_e^{-1}\|\lr{v_h}\|_e^2\big)\\
&\lesssim|v_h|^2_{1,h}\lesssim\|\varepsilon_\mathbb{C}^h(v_h)\|_\mathbb{A}^2,
\end{aligned}\end{equation*}
which completes the proof.
\end{proof} 

\section{Implementation and numerical experiment}\label{secNE}
The method \eqref{dismix} can be implemented using the hybridization technique.
Consider the multiplier space
\begin{align*}
    M_h=\{\mu_h: \mu_h|_e\in\mathcal{P}_{r+3}(e,\mathbb{R}^2)\text{ for all }e\in\CE^o_h\},
\end{align*}
and the broken discrete stress space
\begin{align*}
    \Sigma_h^{-1}&=\{\tau_h\in L^2(\Omega,\mathbb{S}): \tau_h|_T\in\mathcal{P}_{r+3}(T)\text{ for all }T\in\Th\}.
\end{align*}
The hybridized mixed method seeks $(\tilde{\sigma}_h,\tilde{u}_h,\lambda_h)\in\Sigma_h^{-1}\times U_h\times M_h$ such that
\begin{equation}\label{hdismix}
    \begin{aligned}
    (\mathbb{A}\tilde{\sigma}_h,\tau_h)+\sum_{T\in\Th}(\divg\tau_h,\tilde{u}_h)_T+\sum_{e\in\CE^o_h}\int_e\lambda_h\cdot\lr{\tau_h}n_eds&=\langle\tau_hn,g_D\rangle_{\Gamma_D},\\
    \sum_{T\in\Th}(\divg\tilde{\sigma}_h,v_h)_T&=(f,v_h),\\
    \sum_{e\in\CE^o_h}\int_e\mu_h\cdot\lr{\tilde{\sigma}_h}n_eds&=0,
\end{aligned}
\end{equation}
for all $\tau\in\Sigma^{-1}_h,$ $v_h\in U_h$, $\mu_h\in M_h$.
In fact, \eqref{hdismix} is a hybridized version of \eqref{dismix}, i.e., $\tilde{\sigma}_h=\sigma_h, \tilde{u}_h=u_h$, see \cite{GWX2019,ArnoldBrezzi1985}. Because $\Sigma^{-1}_h,$ $U_h$, $M_h$ are completely broken, it is straightforward to construct their local basis. In matrix notation, \eqref{hdismix} reads
\begin{equation}\label{matrixh}
    \begin{pmatrix}A&B\\B^\top&O\end{pmatrix}\begin{pmatrix}X\\ \Lambda\end{pmatrix}=\begin{pmatrix}F\\O\end{pmatrix},
\end{equation}
where $O$ is a zero matrix or vector, $X$ and $\Lambda$ are vectors corresponding to the coordinates of $(\tilde{\sigma}_h,\tilde{u}_h)$ and $\lambda_h,$ respectively.

Due to the discontinuity of $\Sigma_h^{-1}$ and $U_h$, the matrix $A$ is block diagonal and easily invertible. Hence solving \eqref{matrixh} is equivalent to solving the smaller Schur complement system 
\begin{align}\label{spsd}
    B^\top A^{-1}B\Lambda=B^\top A^{-1}F.
\end{align}
Here $B^\top A^{-1}B$ is a sparse and positive semi-definite matrix and the size of $B^\top A^{-1}B$ is much smaller than \eqref{hdismix} or \eqref{dismix}. However,  the Schur complement $B^\top A^{-1}B$ has a small kernel provided $\Th$ has singular vertices and/or pure traction boundary condition ($\Gamma_D=\emptyset$) is considered. The key point is that such kernel could be easily resolved by iterative methods such as the preconditioned conjugate gradient method. An optimal preconditioner for the Schur complement system \eqref{spsd} is presented in \cite{GWX2019}.

In the experiment, let $\Omega=[-1,1]^2\backslash([0,1]\times[-1,0])$ be the L-shaped domain.
Let $(r,\theta)$ be the polar coordinate with respect to the origin, where $0\leq\theta\leq\omega=\frac{3\pi}{2}$. Let
\begin{equation*}
\begin{aligned}
\Phi_1(\theta)&=\begin{pmatrix}((z + 2)(\lambda + \mu) + 4\mu
)
\sin(z\theta)- z(\lambda + \mu) \sin((z-2)\theta)\\
z(\lambda + \mu)
(
\cos(z\theta)- \cos((z- 2)\theta))\end{pmatrix},\\
    \Phi_2(\theta)&=\begin{pmatrix}z(\lambda + \mu)(\cos((z -2)\theta)- \cos(z\theta))\\
    -((2-z)(\lambda + \mu) + 4\mu) \sin(z\theta)- z(\lambda + \mu) \sin((z- 2)\theta)\end{pmatrix},
    \end{aligned}
\end{equation*}
and
\begin{equation*}
\begin{aligned}
    \Phi(\theta)&=\{z(\lambda+ \mu) \sin((z-2)\omega) + ((2-z)(\lambda + \mu) + 4\mu)\sin(z\omega)\}\Phi_1(\theta)\\
    &\quad-z(\lambda+ \mu)(\cos((z- 2)\omega)- \cos(z\omega))\Phi_2(\theta).
    \end{aligned}
\end{equation*}
where $z\in(0,1)$ is a root of $(\lambda+3\mu)^2\sin^2(z\omega) = (\lambda+\mu)^2z^2\sin^2(\omega)$.
The most singular part of the solution to \eqref{ctsLE} behaves like $r^z\Phi(\theta)$ in the neighborhood of $(0,0)$, see, e.g., \cite{Grisvard1992}. Therefore we choose 
\begin{equation*}
    u(r,\theta)=\frac{1}{(\lambda+\mu)^2}(x_1^2-1)(x_2^2-1)r^z\Phi(\theta)
\end{equation*}
as the exact solution in the test problem. The boundary condition is based on pure displacement ($\Gamma_N=\emptyset$). The Lam\'e constants are $\lambda=10^4$ and $\mu=1.$ The method \eqref{dismix} or \eqref{hdismix} is implemented using the package iFEM \cite{iFEM} in Matlab 2019a. We start with the initial mesh in Figure \ref{mesh} and set the marking parameter $\theta=0.3$. The algebraic system \eqref{spsd} is solved by the conjugate gradient method preconditioned by the incomplete Cholesky decomposition. Numerical results are presented in Figure \ref{ErrorCurve}, where nt denotes the number of triangles.

It can be observed from Figure \ref{mesh}(right) that the adaptive algorithm \ref{AMFEM} captures the corner singularity. Figure \ref{ErrorCurve} shows that Algorithm \ref{AMFEM} has optimal and robust rate of convergence with respect to very large Lam\'e constant $\lambda$ starting from coarse initial grid, which validates our convergence and complexity analysis. 
\begin{figure}[tbhp]
\centering
\includegraphics[width=12cm,height=5.0cm]{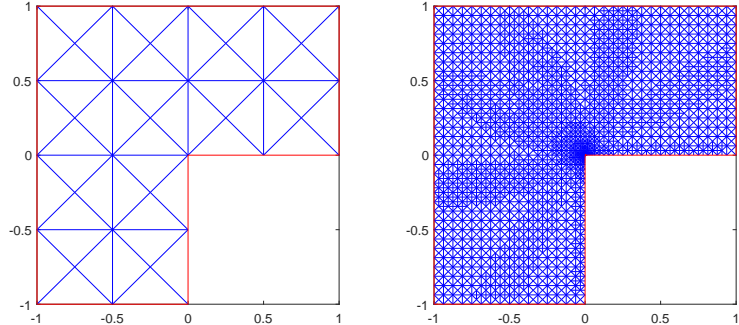}
\caption{(left)Initial grid. (right)Adaptive grid, 5290 elements.}
\label{mesh}
\end{figure}

\begin{figure}[tbhp]
\centering
\includegraphics[width=8cm,height=5cm]{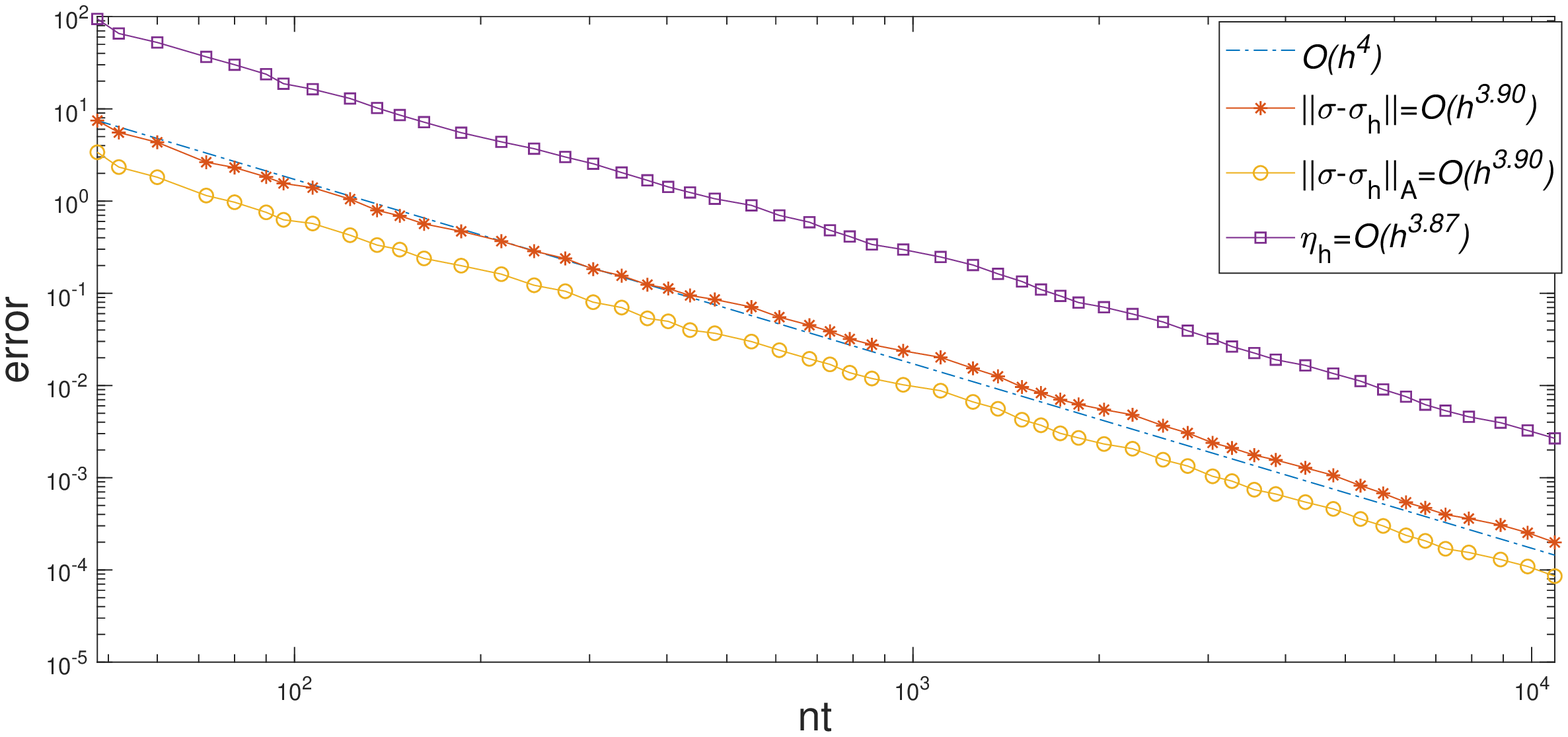}
\caption{Error curve}
\label{ErrorCurve}
\end{figure}


\section*{Acknowledgements}
The author would like to thank Dr.~Shihua Gong for generously sharing his Matlab code and comments on iterative methods.

\providecommand{\bysame}{\leavevmode\hbox to3em{\hrulefill}\thinspace}
\providecommand{\MR}{\relax\ifhmode\unskip\space\fi MR }
\providecommand{\MRhref}[2]{%
  \href{http://www.ams.org/mathscinet-getitem?mr=#1}{#2}
}
\providecommand{\href}[2]{#2}

\end{document}